\documentclass[12pt,bezier]{article}
\usepackage{times}
\usepackage{booktabs}
\usepackage{pifont}
\usepackage{floatrow}
\usepackage{caption}
\usepackage{mathrsfs}
\usepackage{amsmath}
\usepackage{amsfonts,amsthm,amssymb,mathrsfs,bbding}
\usepackage{txfonts}
\usepackage{graphics,multicol}
\usepackage{graphicx}
\usepackage{color}
\usepackage{longtable}

\newtheorem{case}{Case}
\newtheorem{subcase}{Subcase}
\usepackage[colorlinks,
            linkcolor=black,
            anchorcolor=black,
            citecolor=black]{hyperref}
\usepackage{url}
\usepackage{caption}
\usepackage{cite}
\usepackage{latexsym,bm}
\usepackage{indentfirst}
\usepackage{color}
\usepackage[T1]{fontenc}
\usepackage[utf8]{inputenc}
\usepackage{authblk}
\evensidemargin=\oddsidemargin\topmargin=-1.5cm

\newtheorem{theorem}{Theorem}[section]

\newtheorem{lemma}{Lemma}[section]
\newtheorem{cor}{Corollary}[section]

\newtheorem{conj}{Conjecture}
\newtheorem{defi}{Definition}[section]

\theoremstyle{definition}

\addtocounter{section}{0}

\begin{document}

\title{\bf\Large Extremal graphs for the sum of the first two largest signless Laplacian eigenvalues
\footnote{ Corresponding author Email:\ zhibindu@126.com (Z.-B. Du).\\
		This research is supported by  National Natural Science
		Foundation of China (No.12271362).   }}
\author[1]{Zi-Ming Zhou}
\author[2]{Zhi-Bin Du}
\author[1]{Chang-Xiang He}
\affil[1]{College of Science, University of Shanghai for Science and Technology, Shanghai, P. R. China}
\affil[2]{School of Artificial Intelligence, South China Normal University, Foshan, Guangdong 528225,  P. R. China}
 
\renewcommand\Authands{ and }

\date{}

\maketitle

\begin{abstract}
For a graph $G$, let $S_2(G)$ be the sum of the first two largest signless Laplacian eigenvalues of $G$, and $f(G)=e(G)+3-S_2(G)$. Very recently, Zhou, He and Shan proved that $K^+_{1,n-1}$ (the star graph with an additional edge)  is the unique graph  with minimum value of $f(G)$ among graphs on $n$ vertices. In this paper, we prove that $K^+_{1,e(G)-1}$ is the unique graph  with minimum value of $f(G)$ among graphs with $e(G)$ edges.  
\medskip

\begin{flushleft}
\end{flushleft}

\end{abstract}

\section{Introduction}

Let $G=(V(G),E(G))$ be a  simple  graph, the number of vertices and edges are denoted by $n(G)$ (or $n$ for short) and $e(G)$, respectively.
Let $A(G)$ be the adjacency matrix of $G$, and $D(G)$ be the diagonal matrix of vertex degrees of $G$.
The Laplacian matrix and the signless Laplacian matrix of $G$ are defined as $L(G)=D(G)-A(G)$ and $Q(G)=D(G)+A(G)$, respectively.
The eigenvalues of $L(G)$ and $Q(G)$ are called  Laplacian eigenvalues and  signless Laplacian eigenvalues of $G$, respectively,
and are denoted by $\mu_1(G)\geq\mu_2(G)\geq\cdots\geq\mu_n(G)$
(or $\mu_1\geq\mu_2\geq\cdots\geq\mu_n$ for short)
and $q_1(G)\geq q_2(G)\geq\cdots\geq q_n(G)$
(or $q_1\geq q_2\geq\cdots\geq q_n$ for short), respectively.

Consider $M(G)$ as a matrix of a graph $G$ of order $n$ and let $k$ be a natural number such that $1\leq k\leq n$.
A general question related to $G$ and $M(G)$ can be raised: ``How large can the sum of the first $k$ largest eigenvalues of $M(G)$ be?''
Usually, solving the cases $k=1, n-1$ and $n$ are simple, but the general case for other $k$ is not easy to be solved. The natural next case to be studied is $k=2$, and some work has been recently done in order to prove this case, i.e., see Ebrahimi \textit{et al.} \cite{Ebrahimi} for the adjacency matrix, Haemers \textit{et al.} \cite{Haemers} for the Laplacian matrix, Zhou \textit{et al.} \cite{ZHS} for the signless Laplacian matrix.

After clarifying the upper bound of the sum of the first two largest eigenvalues, the next question is: ``What does the corresponding extremal graphs look like?''
In \cite{LG}, Li and Guo proposed the full Brouwer's Laplacian spectrum conjecture and proved that the conjecture holds for $k=2$ with a unique extremal graph class.

\begin{theorem}[\cite{LG}]\label{them1}
Let $G$ be a graph on $n$ vertices. Then
$$
\mu_1+\mu_2\leq e(G)+3
$$ 
with equality if and only if $G\cong G(s,n-2-s)$ with $0\leq s\leq n-3$. The graph class $G(s,n-2-s)$ with $0\leq s\leq n-3$ is depicted in Fig. \ref{Extremal graph}.
\end{theorem}

Actually, for any graph $G$ with $e(G)$ edges, $G(s,\frac{e(G)-1-s}{2})\cong G(s,n-2-s)$ is also the unique graph class which satisfies $\mu_1+\mu_2=e(G)+3$,
the edge version of Theorem \ref{them1} as below is straightforward.

\begin{theorem}\label{them2}
Let $G$ be a graph with $e(G)$ edges. Then
$$
\mu_1+\mu_2\leq e(G)+3
$$ 
with equality if and only if $G\cong G(s,\frac{e(G)-1-s}{2})$, where $0\leq s\leq e(G)-3$, and $s$ and $e(G)$ have different parities.
\end{theorem}

\begin{figure}[htbp!]
\setlength{\unitlength}{0.8pt}
\begin{center}
\begin{picture}(337.1,112.4)
\put(134.1,29.0){\circle*{6}}
\put(63.8,87.7){\circle*{6}}
\put(118.9,87.7){\circle*{6}}
\put(18.1,71.8){\circle*{6}}
\put(18.1,13.8){\circle*{6}}
\put(47.9,29.0){\circle*{6}}
\put(45,0.0){\makebox(0,0)[tl]{$G(s,n-2-s)$}}
\put(91.4,87.7){\circle*{4}}
\put(77.6,87.7){\circle*{4}}
\put(105.9,87.7){\circle*{4}}
\put(18.1,43.5){\circle*{4}}
\put(18.1,29.0){\circle*{4}}
\put(18.1,58.0){\circle*{4}}
\qbezier(134.1,29.0)(91.0,29.0)(47.9,29.0)
\put(0.0,52.2){\makebox(0,0)[tl]{$s$}}
\put(65,112.4){\makebox(0,0)[tl]{$n-2-s$}}
\qbezier(47.9,29.0)(55.8,58.4)(63.8,87.7)
\qbezier(63.8,87.7)(99.0,58.4)(134.1,29.0)
\qbezier(47.9,29.0)(83.4,58.4)(118.9,87.7)
\qbezier(118.9,87.7)(126.5,58.4)(134.1,29.0)
\qbezier(47.9,29.0)(33.0,50.4)(18.1,71.8)
\qbezier(47.9,29.0)(33.0,21.4)(18.1,13.8)

\put(250.9,29.7){\circle*{6}}
\put(337.1,29.7){\circle*{6}}
\put(323.4,88.5){\circle*{6}}
\put(221.9,16.0){\circle*{6}}
\put(221.1,73.2){\circle*{6}}
\qbezier(250.9,29.7)(294.0,29.7)(337.1,29.7)
\put(221.9,29.7){\circle*{4}}
\put(221.1,59.5){\circle*{4}}
\put(221.1,45.0){\circle*{4}}
\qbezier(250.9,29.7)(287.1,59.1)(323.4,88.5)
\qbezier(323.4,88.5)(330.2,59.1)(337.1,29.7)
\qbezier(250.9,29.7)(236.4,22.8)(221.9,16.0)
\qbezier(250.9,29.7)(236.0,51.5)(221.1,73.2)
\put(180,52.2){\makebox(0,0)[tl]{$n-3$}}
\put(260.3,0.7){\makebox(0,0)[tl]{$K^+_{1,n-1}$}}
\end{picture}
\end{center}
\caption{Extremal graphs}
\label{Extremal graph}
\end{figure}
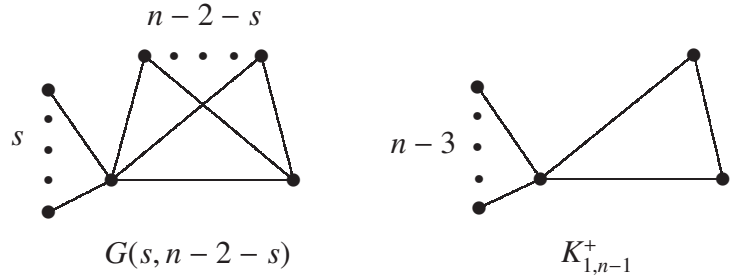


Let $S_2(G)$ be the sum of the first two largest signless Laplacian eigenvalues of $G$, and set
$$
f(G):=e(G)+3-S_2(G).
$$
Very recently, Zhou, He and Shan proved in \cite{ZHS} that $K^+_{1,n-1}$ (the star graph with an additional edge, see Fig. \ref{Extremal graph})  is the unique graph with minimum $f(G)$ among the graphs $G$ on $n$ vertices.

\begin{theorem}[\cite{ZHS}]\label{them3}
Let $G$ be a graph on $n$ vertices. Then
$$
f(G)\geq f(K^+_{1,n-1})
$$ 
with equality if and only if $G\cong K^+_{1,n-1}$.
\end{theorem}

For the signless Laplacian spectrum, since there is no graph satisfying $S_2(G)=e(G)+3$, the edge version of Theorem \ref{them3} is not straightforward as the Laplacian spectrum.

In this paper, our purpose is to present the edge version of Theorem \ref{them3}, more precisely, we are able to 
prove that $K^+_{1,n-1}$ is also the unique graph  with minimum $f(G)$ among the graphs $G$ having $e(G)$ edges.  

\begin{theorem}\label{them4}
Let $G$ be a graph with $e(G) \ge 4$ edges. Then 
$$
f(G)\geq f(K^+_{1,e(G)-1})
$$ 
with equality if and only if $G\cong K^+_{1,e(G)-1}$.
\end{theorem}

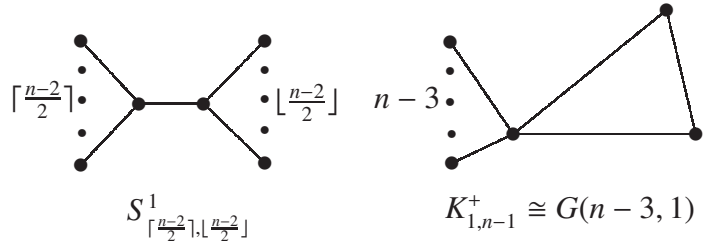
\begin{figure}[htbp!]
\setlength{\unitlength}{0.8pt}
\begin{center}
\begin{picture}(294.5,88.5)


\put(232.7,29.0){\circle*{6}}
\put(319.0,29.0){\circle*{6}}
\put(305.2,87.7){\circle*{6}}
\put(203.0,72.5){\circle*{6}}
\put(203.7,15.2){\circle*{6}}
\put(203.7,29.0){\circle*{4}}
\put(203.0,58.7){\circle*{4}}
\put(203.0,44.2){\circle*{4}}
\put(167,51.5){\makebox(0,0)[tl]{$n-3$}}
\put(200.2,0.0){\makebox(0,0)[tl]{$K^+_{1,n-1}\cong G(n-3,1)$}}

\qbezier(232.7,29.0)(275.9,29.0)(319.0,29.0)
\qbezier(232.7,29.0)(269.0,58.4)(305.2,87.7)
\qbezier(305.2,87.7)(312.1,58.4)(319.0,29.0)
\qbezier(232.7,29.0)(218.2,22.1)(203.7,15.2)
\qbezier(232.7,29.0)(217.9,50.8)(203.0,72.5)
\qbezier(55.8,43.5)(71.1,43.5)(86.3,43.5)

\put(55.8,43.5){\circle*{6}}
\put(86.3,43.5){\circle*{6}}
\put(115.3,73.2){\circle*{6}}
\put(115.3,15.2){\circle*{6}}
\put(28.3,14.5){\circle*{6}}
\put(28.3,73.2){\circle*{6}}
\put(115.3,59.5){\circle*{4}}
\put(115.3,31.2){\circle*{4}}
\put(115.3,45.7){\circle*{4}}
\put(29.0,59.5){\circle*{4}}
\put(29.0,29.7){\circle*{4}}
\put(29.0,45.0){\circle*{4}}
\put(-5,54.4){\makebox(0,0)[tl]{$\lceil \frac{n-2}{2}\rceil$}}
\put(120,52.9){\makebox(0,0)[tl]{$\lfloor \frac{n-2}{2}\rfloor$}}
\put(50,0.7){\makebox(0,0)[tl]{$S^1_{\lceil \frac{n-2}{2}\rceil,\lfloor \frac{n-2}{2}\rfloor}$}}

\qbezier(55.8,43.5)(42.1,58.4)(28.3,73.2)
\qbezier(55.8,43.5)(42.1,29.0)(28.3,14.5)
\qbezier(86.3,43.5)(100.8,58.4)(115.3,73.2)
\qbezier(86.3,43.5)(100.8,29.4)(115.3,15.2)
\end{picture}
\end{center}
\caption{Extremal graphs with $c(G)=0$ and $1$}
\label{Extremal graph2}
\end{figure}


The dimension of cycle space of $G$ is defined as $c(G)=e(G)-n(G)+\omega(G)$, where $\omega(G)$ denotes the number of connected components of $G$. It is easy to see when $G$ is connected, $c(G)=e(G)-n(G)+1$, $G$ is a tree if $c(G)=0$; $G$ is a unicyclic graph if $c(G)=1$.

Guan, Zhai and Wu proved in \cite{Guan} that $S^1_{\lceil \frac{n-2}{2}\rceil,\lfloor \frac{n-2}{2}\rfloor}$ (see Fig. \ref{Extremal graph})  is the unique extremal tree that attains the bound of $\mu_1+\mu_2$, which also minimum $f(G)$ among the tree graphs with order $n\geq 4$.

Actually, for any connected graph $G$ with $n$ vertices and $1\leq c(G)\leq n-2$, $G(n-2-c(G),c(G))\cong G(s,n-2-s)$ is also the unique graph class which satisfies $\mu_1+\mu_2=e(G)+3$.
For the signless Laplacian spectrum, since there is no graph satisfying $S_2(G)=e(G)+3$, this conclusion is not straightforward as the Laplacian spectrum.
We present a conjecture on the uniqueness of the extremal graph.

\begin{conj}\label{conj}
Among all connected graphs with $n$ vertices and $1\leq c(G)\leq n-2$, $G(n-2-c(G),c(G))\cong G(s,n-2-s)$ is the unique graph with maximal value of $S_2(G)$.
\end{conj}

The Theorem \ref{them3} shows that Conjecture \ref{conj} is true for graphs with $c(G)=1$.

The organization of the remaining of the paper is as follows.
In Section 2, we give some lemmas which are useful. In Section 3, we give the proof of Theorem \ref{them4}.

\section{Preliminaries}

We begin this section with some definitions and notations. 
For convenience, we would calculate $S_2(G)$ by virtue of the additive compound matrix, which is introduced via the following three definitions.




\begin{defi} (Lexicographically ordered $k$-sets)
Fix a number $k$ between $1$ and $n$, inclusive. Consider all possible sets formed by $k$ distinct numbers between $1$ and $n$, and order them in a  lexicographic order as 
$$
S_1< S_2< \ldots< S_{\tbinom{n}{k}}.
$$
\end{defi}


\begin{defi} (Compound matrix)
Let $A$ be a complex matrix  of order $n$, and $1 \le k \le n$.
\begin{itemize}
    \item 
Let $A_{i,j}$ be the submatrix of $A$ of order $k$ induced by the rows indexed by $S_i$ and the columns indexed by $S_j$, where $1\le i,j \le \tbinom{n}{k}$.

    \item
Define $C_k(A)$ as  the $k$-th   compound matrix of $A$, which is the matrix of size $ \tbinom{n}{k} \times \tbinom{n}{k}$ whose $(i,j)$-th entry is the determinant of  $A_{i,j}$.
\end{itemize}
\end{defi}

\begin{defi} (Additive compound matrix)
Given $1 \le k \le n$.
The matrix $C_k(I_n+tA)$ is an $ \tbinom{n}{k} \times \tbinom{n}{k}$ matrix,
each entry of which is a polynomial in $t$ of degree at most $k$, and the $k$-th   additive compound matrix is defined as
$$
\Delta_k(A)= \frac{d}{dt}C_k(I_n+tA)\big\vert_{t=0},
$$
where $I_n$ represents the identity matrix of order $n$.
\end{defi}

For additive compound matrix, the following lemma is  well-known.

\begin{lemma}[\cite{MF}]\label{MF}
If $\lambda_1,\ \lambda_2,\ldots, \lambda_n$ are all the eigenvalues of $A$,
then the eigenvalues of $\Delta_k(A)$ are the $\tbinom{n}{k}$  distinct sums of the $\lambda_i$ taken $k$ at a time. That is to say, each eigenvalue of $\Delta_k(A)$ is of the form 
$$
\lambda_{i_1} +\lambda_{i_2}+\dots +\lambda_{i_k}
$$ 
with $1 \le i_1 < i_2 < \dots < i_k \le n$.
\end{lemma}


Let us recall the bounds of $f(K^+_{1,n-1})$ obtained in \cite{ZHS}.

\begin{lemma}[\cite{ZHS}]\label{sn+}
Let $G$ be isomorphic to $K^+_{1,n-1}$. Then 
$$
\frac{1.3}{n}<f(G)<\frac{1.5}{n} 
$$
for any $n\geq7$.
\end{lemma}

Since $e(K^+_{1,n-1})=n(K^+_{1,n-1})$, the following result is straightforward.

\begin{cor}\label{esn+}
Let $G$ be isomorphic to $K^+_{1,e(G)-1}$. Then 
$$
\frac{1.3}{e(G)}<f(G)<\frac{1.5}{e(G)}
$$
for any $e(G)\geq7$.
\end{cor}

\begin{lemma}[\cite{Fritscher}]\label{Ltree}
Let $G$ be a tree on $n$ vertices.
Then 
$$
f(G) > \frac{2}{n}, 
$$
or equivalently, 
$$
f(G) > \frac{2}{e(G)+1}.
$$
\end{lemma}




For any square matrix $M$, we use $P(M,x)$ to denote the characteristic polynomial of $M$. For a graph $G$, we write the characteristic polynomial of its signless Laplacian matrix $P(Q(G),x)$ as $P_Q(G,x)$ for short.

\begin{lemma}\label{ff}
For any positive integer $a \ge 3$, the value $f(K^+_{1,a})$ is monotonically decreasing with respect to $a$.
\end{lemma}

\begin{proof}
When $a \ge 3$, we can partition the vertices of $K^+_{1,a}$ into there parts, the first one consists of the two vertices of degree $2$, the second one contains uniquely the vertex of degree $a$, and the third one is formed by the $a-2$ vertices of degree $1$. 
The corresponding quotient matrix is
$$
Q_a^\pi=
\begin{pmatrix}
3&1&0\\
2&a&a-2\\ 
0&1&1
\end{pmatrix}.
$$

By the  equitable partition theory, we can get that the characteristic polynomial of $K^+_{1,a}$ is
$$
P_Q(K^+_{1,a},x)=P(Q_a^\pi,x)(x-1)^{a-2},
$$
where 
$$
P(Q_a^\pi,x) = x^3 - (a+4)x^2 + 3(a+1)x -4.
$$

From 
$$
P(Q_a^\pi,1) = 2 (a - 2) > 0 ~~~~\mbox{and} ~~~~P(Q_a^\pi,3) = -4 < 0,
$$
we know that the first two largest roots of $P(Q_a^\pi,x) = 0$ are both larger than $1$, thus the first two largest roots of $P_Q(K^+_{1,a},x) = 0$ is actually 
the first two largest roots of $P(Q_a^\pi,x) =0$.
As a consequence of Lemma \ref{MF} with $k=2$, $S_2 (K^+_{1,a})$ is equal to the largest eigenvalue of $\Delta_2(Q_a^\pi)$. 

It is not hard to get the second additive compound matrix of $Q_a^\pi$ as
$$
\Delta_2(Q_a^\pi)=
\begin{pmatrix}
a+3&a-2&0\\
1&4&1\\ 
0&2&a+1
\end{pmatrix}.
$$
Set
$$
P_a(x) :=  P(\Delta_2(Q_a^\pi),x+a+4)=x^3+(a+4)x^2+3(a+1)x+4.
$$
It leads to
$$
P_{a+1}(x) -P_a(x) = x( x +3 ).
$$

From $P_a(x) > 0$ for any $x \ge 0$, and $P_a(-1)=-2(a-2) < 0$, 
the largest root of $P_a(x) = 0$ is in the interval $(-1,0)$, for any positive integer $a \ge 3$.
So $P_{a+1}(x) -P_a(x)$ is always negative when $x\in (-1,0)$,
which indicates that the largest root of $P_{a+1}(x)=0$ is larger than the largest root of $P_{a}(x)=0$, equivalently,
$$
S_2(K^+_{1,a+1}) - a -5 > S_2(K^+_{1,a}) - a -4.
$$

Observe that $e(K^+_{1,a})=a+1$ and $e(K^+_{1,a+1})=a+2$, thus
$$
f(K^+_{1,a}) = a + 4 - S_2(K^+_{1,a}) > a +5 - S_2(K^+_{1,a+1}) = f(K^+_{1,a+1}),
$$
as requested.
\end{proof}

\begin{lemma}[\cite{DCve}]\label{interlace}
Let $G$ be a graph with $n$ vertices and let $G^\prime$ be a graph obtained from $G$ by inserting a new edge into $G$. Then the signless Laplacian eigenvalues of $G$ and $G^\prime$ interlace, that is,
$$q_1(G^\prime)\geq q_1(G)\geq\cdots\geq q_n(G^\prime)\geq q_n(G).$$
\end{lemma}

\begin{lemma}[\cite{KF}]\label{KF}
Let $A$ and $B$ be two real symmetric matrices of size $n$. Then for any $1\leq k\leq n$,
$$\sum^{k}_{i=1}\lambda_{i}(A+B)\leq\sum^{k}_{i=1}\lambda_{i}(A)+\sum^{k}_{i=1}\lambda_{i}(B).$$
\end{lemma}

Apparently, $\lambda_{1}(\Delta_k(A))=\sum^{k}_{i=1}\lambda_{i}(A)$,
then the above lemma can be directly obtained from the fact that
$\lambda_{1}(\Delta_k(A+B))\leq\lambda_{1}(\Delta_k(A))+\lambda_{1}(\Delta_k(B)).$ An immediate result of Lemma \ref{KF} is the following corollary which will be used frequently.

\begin{cor}\label{union}
Let $G$ be a graph of order $n$ such that there exist edge disjoint subgraphs of $G$, say $G_1,\cdots,G_r$, with $E(G)=\bigcup^r_{i=1}E(G_i)$. Then $S_k(G)\leq\sum^{r}_{i=1}S_k(G_i)$ for any integer $k$ with $1\leq k\leq n$.
\end{cor}

\begin{lemma}\label{ksubgraph}
If $G$ is a graph with a nonempty subgraph $H$ for which $S_k(H)\leq e(H)$, then $S_k(G)<e(G)$.
\end{lemma}

\begin{proof}
Assume that $G$ is a counterexample with a minimum possible number of edges.
By Corollary \ref{union}, we have $e(G)\leq S_k(G)\leq S_k(H)+S_k(G-H)$.
This implies that $S_k(G-H)\geq e(G-H)$,
since $e(G)\geq e(H)+e(G-H)$ and $S_k(H)\leq e(H)$,
which contradicts the minimality of $e(G)$.
\end{proof}

\begin{cor}\label{subgraph}
If $G$ is a graph with a nonempty subgraph $H$ for which $S_2(H)\leq e(H)$, then $S_2(G)<e(G)$ and $f(G)>3$.
\end{cor}

\section{Proof of Theorem \ref{them4}}

Now we are ready to present the proof of Theorem \ref{them4}. Let $G$ be a graph with $e(G)$ edges. We partition the proof into two parts.

\begin{case}
The first two largest signless Laplacian eigenvalues come from two different components of $G$.     
\end{case}

Assume that $S_2(G)=q_1(G_1)+q_1(G_2)$, where $G_1$ and $G_2$ are two components of $G$. It is well-known that the largest signless Laplacian eigenvalue of any graph is no more than the number of edges plus one, e.g., see \cite{Ashraf}. So we have 
$$
S_2(G)=q_1(G_1)+q_1(G_2)\leq e(G_1)+e(G_2)+2 \le e(G)+2.
$$ 
Then we obtain $f(G)\geq 1$, as well as
$f(G)>f(K^+_{1,e(G)-1})$, since $f(K^+_{1,e(G)-1}) < \frac{1.5}{e(G)}$ from Lemma \ref{esn+}.

\begin{case}
The first two largest signless Laplacian eigenvalues are exactly the first two largest signless Laplacian eigenvalues of some components of $G$.     
\end{case}


\begin{subcase}
$G$ is a tree.
\end{subcase}

First suppose that $e(G)=4$. All the trees with $e(G)=4$ are depicted in Fig. \ref{Tree graph}.
By direct calculation, we have $f(T_i)>f(K^+_{1,3}),\ i=1,2,3$.

\begin{figure}[htbp!]
\setlength{\unitlength}{0.8pt}
\begin{center}
\begin{picture}(377.0,60.2)
\put(27.6,31.2){\circle*{6}}
\put(0.0,31.2){\circle*{6}}
\qbezier(27.6,31.2)(13.8,31.2)(0.0,31.2)
\put(58.0,31.2){\circle*{6}}
\qbezier(27.6,31.2)(42.8,31.2)(58.0,31.2)
\put(86.3,31.2){\circle*{6}}
\qbezier(58.0,31.2)(72.1,31.2)(86.3,31.2)
\put(116.7,31.2){\circle*{6}}
\qbezier(86.3,31.2)(101.5,31.2)(116.7,31.2)
\put(174.0,31.2){\circle*{6}}
\put(232.0,31.2){\circle*{6}}
\qbezier(174.0,31.2)(203.0,31.2)(232.0,31.2)
\put(203.0,31.2){\circle*{6}}
\put(261.0,60.2){\circle*{6}}
\qbezier(232.0,31.2)(246.5,45.7)(261.0,60.2)
\put(261.0,31.2){\circle*{6}}
\qbezier(232.0,31.2)(246.5,31.2)(261.0,31.2)
\put(319.0,31.9){\circle*{6}}
\put(377.0,31.9){\circle*{6}}
\qbezier(319.0,31.9)(348.0,31.9)(377.0,31.9)
\put(348.0,31.9){\circle*{6}}
\put(376.3,60.2){\circle*{6}}
\qbezier(348.0,31.9)(362.1,46.0)(376.3,60.2)
\put(376.3,3.6){\circle*{6}}
\qbezier(348.0,31.9)(362.1,17.8)(376.3,3.6)
\put(49.3,0.0){\makebox(0,0)[tl]{$T1$}}
\put(209.5,0.7){\makebox(0,0)[tl]{$T2$}}
\put(338.6,0.7){\makebox(0,0)[tl]{$T3$}}
\end{picture}
\end{center}
\caption{All the trees with $e(G)=4$.}
\label{Tree graph}
\end{figure}
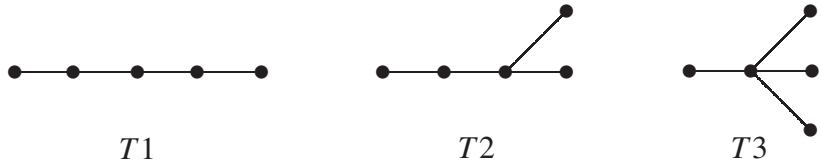


Next suppose that $e(G)=5,6$. On one hand, by direct calculation, we have $f(K^+_{1,4})<\frac{2}{6}$ and $f(K^+_{1,5})<\frac{2}{7}$. On the other hand, from 
Lemma \ref{Ltree}, we have $f(G)> \frac{2}{e(G)+1}$,
then $f(G) > f(K^+_{1,e(G)-1})$ follows.

When $e(G)\geq7$, from Lemmas \ref{Ltree} and \ref{esn+}, we have 
$$
f(G)> \frac{2}{e(G)+1} > \frac{1.5}{e(G)} > f(K^+_{1,e(G)-1}).
$$

\begin{subcase}
$G$ is connected but not a tree.
\end{subcase}

Assume that the order of $G$ is  $n$. Since $G$ is a connected graph but not a tree, we would have $e(G)\geq n$. If $G$ is not isomorphic to $K^+_{1,e(G)-1}$, then we have $f(G) > f(K^+_{1,n-1})$ by Theorem \ref{them3}, and further by Lemma \ref{ff}, we get $f(K^+_{1,n-1}) \geq f(K^+_{1,e(G)-1})$ from $e(G)\geq n$. Combining the two inequalities, it results in $f(G) > f(K^+_{1,e(G)-1})$.

\begin{subcase}
$G$ is not connected.
\end{subcase}

Assume that $H$ is a component of $G$ with $S_2 (G) = S_2 (H)$. We would have
$$
f(G) = e(G) + 3 - S_2 (G) \ge e(H) + 3 - S_2 (H) = f(H). 
$$
As in the above two subcases about connected graphs, we get $f(H) > f(K^+_{1,e(H)-1})$. And from Lemma \ref{ff}, $f(K^+_{1,e(H)-1}) \ge f(K^+_{1,e(G)-1})$
since $e(G) \ge e(H)$. So 
$f(G)>f(K^+_{1,e(G)-1})$ follows.

In conclusion, when $G$ is not isomorphic to $K^+_{1,e(G)-1}$, we always have 
$$
f(G)>f(K^+_{1,e(G)-1}).
$$
The proof of Theorem \ref{them4} is completed.


\end{document}